\documentclass[12pt]{amsart}
\usepackage[margin=1in]{geometry}
\usepackage{amssymb}
\usepackage{amsthm}
\usepackage{amstext}
\usepackage{amsbsy}
\usepackage{amscd}
\usepackage{enumerate}
\usepackage{chngpage}
\usepackage{mathtools}
\usepackage{amsmath}
\usepackage{hyperref}
\usepackage{tikz}
\usepackage{stmaryrd}
\usepackage{color}
\usepackage{array}

\newtheorem{thm}{Theorem}[section]
\newtheorem{lemma}[thm]{Lemma} \newtheorem{cor}[thm]{Corollary}
\newtheorem{prop}[thm]{Proposition}
\newtheorem{question}[thm]{Question}

\newtheorem{alg}[thm]{Algorithm}
\theoremstyle{definition}
\newtheorem{defn}[thm]{Definition}
\newtheorem{example}[thm]{Example}

\newtheorem{rmk}[thm]{Remark}

\DeclarePairedDelimiter\floor{\lfloor}{\rfloor}
\DeclarePairedDelimiter\ceil{\lceil}{\rceil}

\newcommand\pbin{\genfrac{[}{]}{0pt}{}}

\newcommand{\Z}{\mathbb Z}
\newcommand{\N}{\mathbb N}
\newcommand{\R}{\mathbb R}
\newcommand{\C}{\mathbb C}
\newcommand{\Q}{\mathbb Q}
\newcommand{\F}{\mathbb F}

\newcommand{\OO}{\mathcal O}

\newcommand{\mb}[1]{\mathbb{#1}}

\begin{document}
	\title{Lower bounds for the number of subrings in $\Z^n$}
	
	\author{Kelly Isham}
	
	\date{}
	
	
	\maketitle 
	
	\begin{abstract}
		Let $f_n(k)$ be the number of subrings of index $k$ in $\Z^n$. We show that results of Brakenhoff imply a lower bound for the asymptotic growth of subrings in $\Z^n$, improving upon lower bounds given by Kaplan, Marcinek, and Takloo-Bighash. Further, we prove two new lower bounds for $f_n(p^e)$ when $e \ge n-1$. Using these bounds, we study the divergence of the subring zeta function of $\Z^n$ and its local factors. Lastly, we apply these results to the problem of counting orders in a number field.
	\end{abstract}

	\let\thefootnote\relax
	
	\let\thefootnote\relax
	\section{Introduction} \label{intro}
	A \emph{subring} of $\mb{Z}^n$ is a sublattice that contains the multiplicative identity $(1,1,\ldots, 1)$ and is closed under componentwise multiplication. Let $f_n(k)$ denote the number of subrings of $\mb{Z}^n$ of finite index $k$. The \emph{subring zeta function} of $\Z^n$ is given by 
	$$
	\zeta_{\Z^n}^R(s) = \sum_{\substack{S \text{ subring of }\\ \text{finite index in } \Z^n}} [\Z^n:S]^{-s} = \sum_{k=1}^\infty f_n(k) k^{-s}.
	$$
	 There is a wider class of zeta functions for groups or rings, which are defined in \cite{gss}. For an extensive survey about these zeta functions, see \cite{ dusautoy}. Let $\prod_p$ denote the product over all primes. The subring zeta function has an Euler product
	$$
	\zeta_{\Z^n}^R(s) = \prod_p \zeta_{\Z^n, \, p}^R(s)
	$$
	with local factors
	$$
	\zeta_{\Z^n, \,p}^R(s) =\sum_{e=0}^\infty f_n(p^e)p^{-es}.
	$$
	Let $\Z_p$ denote the $p$-adic integers. Note that $\zeta_{\Z^n, \,p}^R(s)= \zeta_{\Z_p^n}^R(s)$ where the $p^{-es}$ coefficient in the latter zeta function is the number of $\Z_p$-subalgebras of index $p^e$ in $\Z_p^n$.

	 When $n \leq 4$, there are explicit closed formulas for $\zeta_{\Z^n}^R(s)$, which we summarize in the following theorem. 
	
	\begin{thm}\label{zeta_closed_form}
		We have
		\begin{align*}
		\zeta_{\Z^2}^R(s) &= \zeta(s)\\
		\zeta_{\Z^3}^R(s) &= \frac{\zeta(3s-1)\zeta(s)^3}{\zeta(2s)^s}\\
		\zeta_{\Z^4}^R(s) &= \prod_p \frac{1}{(1-p^{-s})^2(1-p^{2-4s})(1-p^{3-6s})}\bigg( 1 + 4p^{-s}\\
		&+2p^{-2s} +(4p-3)p^{-3s} + (5p-1)p^{-4s} + (p^2-5p)p^{-5s} \\
		&+ (3p^2-4p)p^{-6s}-2p^{2-7s} - 4p^{2-8s} - p^{2-9s}\bigg).
		\end{align*}
	\end{thm}
	The case $n=2$ follows from the fact that $f_2(k) = 1$ for all positive integers $k$. The case $n=3$ is originally due to Datskovsky and Wright \cite{datskovksywright} and $n=4$ is due to Nakagawa \cite{nakagawa}. Liu \cite{liu} gives combinatorial proofs of these formulas. 
	
	In this paper, we are motivated by three broad questions. These questions have been actively studied by several other authors and we summarize some known results below.
	
	\begin{question}\label{fn_ques}
		For each fixed $n$ and $e$, what is $f_n(p^e)$ as a function of $p$?
	\end{question}
	For each fixed $n \le 4$ and $e \ge 0$, comparing the $p^{-es}$ coefficients of the expression in Theorem \ref{zeta_closed_form} shows that $f_n(p^e)$ is a polynomial in $p$.  Liu \cite{liu} gives explicit formulas for $f_n(p^e)$ for fixed $e \leq 5$ and $n >0$, which are all polynomial in $p$. His formula for $e= 5$ has a small error that Atanasov, Kaplan, Krakoff, and Menzel \cite{akkm} correct. Further, the authors extend Liu's work by giving explicit polynomial formulas for $f_n(p^e)$ when $e \in \{6,7,8\}$ and $n >0$. These are the only exact formulas that are known; it is not even known if $f_n(p^e)$ will always be polynomial in $p$. In this paper, we discuss known and new results about lower bounds for $f_n(p^e)$.

	\begin{question}\label{asymp_ques}
		What is the asymptotic growth of the number of subrings in $\Z^n$ of index at most $B$?
	\end{question}
	
	\noindent For each fixed $n$, let $$N_n(B) = \#\{S \subset \Z^n \, :\, S \text{ is a subring and } [\Z^n:S] \leq B\} = \sum_{k\leq B} f_n(k).$$ This function counts the number of subrings of index at most $B$ in $\Z^n$. In \cite{kmt}, Kaplan, Marcinek, and Takloo-Bighash give the asymptotic growth of $N_n(B)$ when $n \leq 5$. The cases $n \leq 4$ follow from Theorem \ref{zeta_closed_form} after applying a Tauberian theorem (see e.g. the appendix of \cite{tauberian}). It is important to note that these authors were able to prove an asymptotic formula when $n=5$ even though no closed formula for $\zeta_{\Z^5}^R(s)$ is known. There is not even a conjecture about the asymptotic growth of $N_n(B)$ for $n \ge 6$. To make progress toward an answer to Question \ref{asymp_ques} when $n \ge 6$, Kaplan, Marcinek, and Takloo-Bighash give upper and lower bounds for the asymptotic behavior of $N_n(B)$.
	
	\begin{thm}\cite[Theorem 6]{kmt} \label{kmt_subring}
		\begin{enumerate}
			\item Let $n \leq 5$. There exists a constant $C_n$ so that 
			$$
			N_n(B) \sim C_n B (\log B)^{\binom{n}{2}-1}
			$$
			as $B \rightarrow \infty$.
			\item Let $n > 5$. For any $\epsilon > 0$, 
			$$
			B(\log B)^{\binom{n}{2}-1} \ll N_n(B) \ll_\epsilon B^{\frac{n}{2} - \frac{7}{6} +\epsilon}
			$$
			as $B \rightarrow \infty$. 
		\end{enumerate}
	\end{thm}
	
\noindent	The authors obtain the lower bound in Theorem \ref{kmt_subring}(2) by computing the rightmost pole of the simpler Euler product $\prod_{p} (1 + f_n(p)p^{-s})$ and then applying a Tauberian theorem. In this paper, we show that the results of Brakenhoff \cite{brakenhoff} lead to a new asymptotic lower bound for $N_n(B)$ that improves upon Theorem \ref{kmt_subring}(2).

Observe that Theorem \ref{kmt_subring} implies that $\zeta_{\Z^n}^R(s)$ diverges for all $s\in \C$ such that $\Re(s) \le 1$. Thus the strategy Kaplan, Marcinek, and Takloo-Bighash employ also gives a partial answer to the question: what is the abscissa of convergence of $\zeta_{\Z^n}^R(s)$? Recall that the \emph{abscissa of convergence} of a Dirichlet series $D(s)$ is the unique $\sigma \in \R\cup \{\pm \infty\}$ so that $D(s)$ diverges for all $s$ with $\Re(s) < \sigma$ and converges for all $s$ with $\Re(s) > \sigma$. We will show that results of Brakenhoff \cite{brakenhoff} improve upon the lower bound for the abscissa of convergence of $\zeta_{\Z^n}^R(s)$. 

	
\begin{question} \label{conv_ques}
	What is the abscissa of convergence of $\zeta_{\Z^n , \, p}^R(s)$?
\end{question}

If $\zeta_{\Z^n}^R(s)$ converges, so do each of the local factors. Therefore the abscissa of convergence for $\zeta_{\Z^n}^R(s)$ gives an upper bound for the abscissa of convergence for $\zeta_{\Z^n, \; p}^R(s)$ for each prime $p$. However, not much is known about lower bounds. We will provide a lower bound for the abscissa of convergence of $\zeta_{\Z^n , \, p}^R(s)$.



\subsection{Main results}
In this paper, we give partial results to the three broad questions asked in Section \ref{intro}. First, we provide a new lower bound for the asymptotic growth of $N_n(B)$ that improves upon \cite{kmt} for all $n \ge 7$ (see Table \ref{valuesan}). The main elements in the proof of this theorem come from interpreting results of Brakenhoff \cite{brakenhoff}.

\begin{thm} \label{brakenhoff_lb}
	Fix $n >1$ and let $$a(n) = \max_{0 \leq d \leq n-1} \left(\frac{d(n-1-d)}{(n-1+d)} + \frac{1}{n-1+d}\right).$$ Then
	$
	B^{a(n)} \ll N_n(B)
	$
	as $B \rightarrow \infty$.
\end{thm}

\begin{center}
\begin{table}[!htp]\caption{Values of $a(n)$ for small $n$}\label{valuesan}
	\centering
\begin{tabular}{|c|m{.2in}m{.2in}m{.2in}m{.2in}m{.2in}m{.2in}m{.2in}c|}
	\hline
\rule{0pt}{3ex}	$n$  & 6&7&8&9&10&20&50&100\\[6pt]
		\hline
\rule{0pt}{3ex}	$a(n)$ & 1& $\frac{9}{8}$ &$ \frac{13}{10}$  &$ \frac{16}{11}$ &$\frac{21}{13}$  
	 &$ \frac{89}{27}$  
	 &$ \frac{581}{69}$ 
	 &$ \frac{2379}{140}$  \\[6pt]
	 	\hline
\end{tabular}\end{table}
\end{center}

 Next, we study two different techniques for bounding the number of subrings in $\Z^n$ of index $p^e$. The first technique is an extension of Brakenhoff's \cite{brakenhoff} results and holds more generally for the function that counts subrings of the ring of integers $\OO_K$ in a fixed number field $K$ of degree $n$. We find a lower bound by showing that a special set of subgroups of $\OO_K$ are subrings and then bounding the number of such subgroups. This is detailed in Section \ref{extended_brakenhoff}. The second technique is based off Liu's \cite{liu} work; we bound subrings by counting the number of $n \times n$ matrices with certain properties, see Sections \ref{method_counting} and \ref{fn_bound_sect}.

We summarize our main results, which give partial answers to Questions \ref{fn_ques} and \ref{conv_ques}. First, we provide new lower bounds for $f_n(p^e)$ that hold for all $e \geq n-1$. We use the first technique to obtain Theorem \ref{fn_bd_subgps}. The other two theorems follow from the second technique.

\begin{rmk}
	Unless otherwise stated, $\max_{A\le x \le B}$ means the maximum over all integers $x$ in the range $[A,B]$.
\end{rmk}

\begin{thm}\label{fn_bd_subgps}
	Fix integers $n> 1$ and $e \ge n-1$. For each $t$, set $k = e - t(n-1)$,  $c = \ceil{\frac{e}{t}} - (n-1)$ and $b = \floor{\frac{e}{t}} - (n-1)$. Set $$h(e,n) = \max_{\ceil{\frac{e}{2(n-1)}} \le t \le \floor{\frac{e}{n-1}}} \left(k(n-1) -c^2(k+t-ct) - b^2(ct-k) \right).$$ Then 
	$f_n(p^{e}) \ge  p^{h(e,n)}$.
\end{thm}

\begin{thm} \label{fn_bound_1}
	Suppose that $e \ge n-1$. Let $$b(n,e) = \max_{0 \leq d \leq n-1} \bigg\lfloor\frac{e}{n-1+d}\bigg\rfloor \cdot d (n-1-d).$$ Then $f_n(p^e)\geq 
	p^{b(n,e)}$.
\end{thm}

The lower bounds from Theorem \ref{fn_bound_1} and the work in Section \ref{asymptotic_section} leads to a lower bound for the abscissa of convergence of $\zeta_{\Z^n, \, p}^R(s)$.

\begin{thm}\label{maintheorem_1} Fix $n> 1$. Let 
	$$
	c_7(n) =\max_{0 \le d \le n-1} \frac{d(n-1-d)}{n-1+d}.
	$$
	Then $\zeta_{\Z^n, \, p}^R(s)$ diverges for all $s$ such that $$\Re(s) \le c_7(n).$$
\end{thm}

\subsection{Outline of the paper}

In the rest of this paper, we discuss lower bounds for various functions related to subrings in $\Z^n$ including $f_n(p^e)$, $N_n(B)$, and the abscissa of convergence of $\zeta_{\Z^n, \, p}^R(s)$. In Section \ref{extended_brakenhoff}, we summarize results from Brakenhoff \cite{brakenhoff} and then show how these results lead to better asymptotic lower bounds for $N_n(B)$. Brakenhoff gives a lower bound for $f_n(p^e)$, when $e \in [(n-1), 2(n-1)]$. We extend his method to obtain a lower bound for $f_n(p^e)$ for all $e \ge n-1$. In Section \ref{method_counting}, we introduce Liu's method of counting subrings in $\Z^n$ by counting matrices in Hermite normal form with certain conditions. We then provide an algorithmic method for counting such matrices. In Section \ref{fn_bound_sect}, we prove new lower bounds for $f_n(p^e)$ using the method discussed in Section \ref{method_counting}. In Section \ref{asymptotic_section}, we prove a lower bound for the abscissa of convergence of $\zeta_{\Z^n, \, p}^R(s)$. In Section \ref{order_sect}, we connect the results in this paper to the problem of counting orders in a number field. Finally, in Section \ref{further_q} we discuss further questions.

\section{Extending results of Brakenhoff} \label{extended_brakenhoff}

In his 2009 PhD thesis \cite{brakenhoff}, Brakenhoff studies similar questions to those asked in Section \ref{intro}. Let $\text{Nf}_n$ be the set of number fields of degree $n$. For $K \in \text{Nf}_n$, let $\OO_K$ be the ring of integers of $K$. Brakenhoff considers the function
$$
f(n,m) = \max_{K \in \text{Nf}_n} \#\left\{R \subset \mathcal{O}_K \, : \, R \text{ is a subring of index } m \right\}.
$$
An \emph{order} in $\OO_K$ is a finite index subring of $\OO_K$ that contains the multiplicative identity. Since we assume that all subrings contain the multiplicative identity, the function $f(n,m)$ is also counting orders.

\begin{lemma} \cite[Lemma 5.10]{brakenhoff} \label{brakenhoff_lemma}
	Every additive subgroup $G\subset\mathcal{O}_K$ that satisfies $\Z + m^2 \OO_K \subset G \subset \Z + m\OO_K$ for some integer $m$ is a subring.
\end{lemma}

Using Lemma \ref{brakenhoff_lemma}, Brakenhoff finds a lower bound for $f(n,p^e)$ when $e \in [n-1, 2(n-1)].$ First we recall a definition.

\begin{defn}
	Let $q$ be a prime power. The $q$-binomial coefficient is given by
	$$
	\pbin{n}{k}_q= \prod_{i=0}^{k-1} \frac{q^n-q^i}{q^k-q^i}.
	$$
\end{defn}

\begin{prop}\cite[Page 42]{brakenhoff}  \label{brakenhoff_refined_bound}
	Fix integers $n > 0$ and $d \in [0,n-1]$. Then $$f(n, p^{n-1+d}) \geq \pbin{n-1}{d}_p.$$
\end{prop}
\begin{rmk}\label{degree_bound}
	It is important to note that $\pbin{n-1}{d}_p$ is a polynomial in $p$ of degree $d(n-1-d)$.
\end{rmk}
\begin{rmk}\label{brakenhoff_bound_rmk}
Note that Lemma \ref{brakenhoff_lemma} is still true for each fixed number field $K$. Further, it is still true if we replace $\OO_K$ with $\Z^n$. Therefore the bound in Proposition \ref{brakenhoff_refined_bound} is also a lower bound for $f_n(p^{n-1+d})$ when $d$ is an integer in $[0, n-1]$.
\end{rmk}
 
  In Section \ref{ext_brak_sec}, we extend Proposition \ref{brakenhoff_refined_bound} by providing a lower bound for $f_n(p^e)$ when $e \ge n-1$. This extended lower bound holds for all $K \in \text{Nf}_n$. In Section \ref{fn_bound_sect}, we provide a different technique for bounding $f_n(p^e)$.

\subsection{Asymptotic lower bounds using Brakenhoff's results}
We now show that Proposition \ref{brakenhoff_refined_bound} leads to new a lower bound for the asymptotic growth of subrings in $\Z^n$. The main theorem in this section (Theorem \ref{brakenhoff_lb}) does not appear in \cite{brakenhoff}.
\begin{lemma}\label{div_zeta}
	Fix $n \ge 0$ and $e \geq 0$. Suppose there exists a constant $c_n >0$ and positive integer $a$ so that $f_n(p^e) \geq c_n p^a$. Then $\zeta_{\Z^n}^R(s)$ diverges for all $s$ such that $\Re(s) \leq \frac{a+1}{e}.$
\end{lemma}

The following corollary was not stated in Brakenhoff's thesis.

\begin{cor}\label{brakenhoff_s}
	Fix $n > 1$. Then $\zeta_{\Z^n}^R(s)$ diverges for all $s$ such that $$\Re(s) \leq \max_{0 \leq d \leq n-1} \left(\frac{d(n-1-d)}{n-1+d} + \frac{1}{n-1+d}\right).$$
\end{cor}
\begin{proof}
	By Remarks \ref{degree_bound} and \ref{brakenhoff_bound_rmk}, for each fixed $n  > 0$ and for each integer $d \in [0, n-1]$, 
	$$
	f_n(p^{n-1+d}) \ge p^{d(n-1-d)}.
	$$
	Applying Lemma \ref{div_zeta} gives the result.
\end{proof}

\begin{proof}[Proof of Theorem \ref{brakenhoff_lb}]
This follows from Corollary \ref{brakenhoff_s} and standard Tauberian theorem.
\end{proof}

\noindent Theorem \ref{brakenhoff_lb} is an improvement upon Theorem \ref{kmt_subring} for all $n \ge 7$. 

\subsection{Extending Brakenhoff's lower bound for $f_n(p^e)$} \label{ext_brak_sec}
We begin by sketching the proof of Proposition \ref{brakenhoff_refined_bound}, which relies on Lemma \ref{brakenhoff_lemma}.

Take $m = p$ in Lemma \ref{brakenhoff_lemma} and consider the following set
$$
\{G \subset \OO_K \, : \, \Z+ p^2\OO_K \subset G \subset \Z+ p\OO_K\text{ and } \dim_{\F_p} \left(G / (\Z+p^2\OO_K)\right) = d\},
$$
which is a subset of the set of subrings in $\OO_K$ of index $p^{n-1+d}$. Note that there is a small typo on page 42 of \cite{brakenhoff}; he writes $\dim_{\F_p} \left(G / (\Z+p\OO_K)\right) = d$. Brakenhoff proves that the cardinality of this set is exactly $\pbin{n-1}{d}_p$ by showing the set is in bijection with the set of $\F_p$-vector spaces in $\F_p^{n-1}$ of dimension $d$.

We now generalize this lower bound by relating subgroups satisfying the condition in Lemma \ref{brakenhoff_lemma} to subgroups in $\left(\Z/m\Z\right)^{n-1}$. Let $K$ be a number field of degree $n$ and consider the set 
$$
\{G \subset \OO_K \, : \, \Z+ m^2 \OO_K \subset G \subset \Z+ m \OO_K\}.
$$
By Lemma \ref{brakenhoff_lemma}, every finite index subgroup in this set is a subring of $\OO_K$. Since $\OO_K$ is a free $\Z$-module of rank $n$, $\OO_K$ is additively isomorphic to $\Z^n$. Let  $\{1, v_1, \ldots, v_{n-1}\}$ be a basis for $\OO_K$. There is a bijection between sublattices of $\Z^{n-1}$ and subgroups of $\OO_K$ that contain $\Z$. 

\begin{prop}\label{G_lattice_cond}
	Let $K$ be a degree $n$ number field and let $\{1 , v_1, \ldots, v_{n-1}\}$ be a basis for $\OO_K$. For any $m\ge 1$, $\Z+ m^2 \OO_K \subset G \subset \Z+ m\OO_K$ if and only if $G$ contains the sublattice of $\Z^{n-1}$ spanned by $m^2v_1, \ldots, m^2v_{n-1}$ and $G$ is contained in the sublattice of $\Z^{n-1}$ spanned by $mv_1, \ldots, mv_{n-1}$.
\end{prop}
\begin{proof}
	Observe that $\Z+ m^2\OO_K$ is the subgroup corresponding to the sublattice of $\Z^{n-1}$ spanned by $ m^2v_1, \ldots, m^2v_{n-1}$. Therefore $\Z + m^2 \OO_K\subset G$ if and only if $G$ contains all basis elements of $\Z+ m^2 \OO_K$. The second condition follows from the fact that $\Z+ m \OO_K$ corresponds to the sublattice spanned by $mv_1, \ldots, mv_{n-1}$.
\end{proof}

\begin{thm}\label{subgps_to_subrings}
	Let $K$ be a degree $n$ number field. Fix some $k\in \N$. The number of subgroups $G$ of index $m^{n-1}k$ such that $\Z + m^2 \OO_K \subset G \subset \Z+ m\OO_K$ is equal to the number of subgroups of order $k$ in $\left(\Z/m\Z\right)^{n-1}$.
\end{thm}
\begin{proof}
	By the Lattice Isomorphism Theorem for groups, there is a bijection between subgroups $G\subset \Z^{n-1}$ such that $m^2 \Z^{n-1} \subset G \subset m\Z^{n-1}$ and subgroups $G'\subset m \Z^{n-1} / m^2 \Z^{n-1} \cong \left(\Z/m\Z\right)^{n-1}$. Fix some $G \subset \Z^{n-1}$ satisfying the conditions. Then 
	$$[\Z^{n-1}: G]  = [\Z^{n-1}: m\Z^{n-1}][m\Z^{n-1}:G] = m^{n-1} [m\Z^{n-1}:G].$$
	Set $k = [m\Z^{n-1}:G]$. Let $G'$ be the subgroup of $\left(\Z/m\Z\right)^{n-1}$ corresponding to $G$. Then $k = [m\Z^{n-1}:G] = [\left(\Z/m\Z\right)^{n-1}: G']$. Finally, observe that the number of subgroups of index $k$ in $\left(\Z/m\Z\right)^{n-1}$ is equal to the number of subgroups of order $k$ in $\left(\Z/m\Z\right)^{n-1}$.
\end{proof}

In order to find a lower bound for $f_n(p^e)$, it now suffices to count the number of subgroups in $\left(\Z/m\Z\right)^{n-1}$. Specifically, we will set $m = p^t$ for some $t > 0$.

\begin{defn}
	Let $\lambda = (\lambda_1, \ldots, \lambda_r)$ be a partition. A finite abelian $p$-group has \emph{type} $\lambda$ if $G  \cong \Z/p^{\lambda_1} \Z\times \cdots \times \Z/p^{\lambda_{r}}\Z.$
\end{defn}

Let $\lambda'$ denote the conjugate partition of $\lambda.$ We use the notation $\nu \subseteq \lambda$ if $\nu_i \le  \lambda_i$ for all $i = 1, \ldots, r$. 

\begin{thm} \label{mu_subgps}\cite[Page 1]{stehling}
	Let $\nu \subseteq \lambda$ be partitions. Let $\nu'$ be the conjugate partition of $\nu$ and $\lambda'$ be the conjugate partition of $\lambda.$ The number of subgroups of type $\nu$ in a finite abelian $p$-group of type $\lambda$ is equal to
	$$
	\prod_{j \ge 1} p^{\nu_{j+1}'(\lambda_j' - \nu_j')} \pbin{\lambda_j' - \nu_{j+1}'}{\nu_{j}' - \nu_{j+1}'}_p.
	$$
\end{thm}

Using this theorem, we can give an exact formula for the number of subgroups of type $\nu$ in a finite abelian $p$-group of type $\lambda= (t,t ,\ldots, t)$.
\begin{cor}\label{mu_subgps_cor}
	The number of subgroups of type $\nu$ in $\left(\Z/p^t\Z\right)^{n-1}$ is equal to
	$$
	\prod_{j \ge 1} p^{\nu_{j+1}'((n-1) - \nu_j')} \pbin{(n-1) - \nu_{j+1}'}{\nu_{j}' - \nu_{j+1}'}_p.
	$$
\end{cor}
\begin{proof}
	Set $\lambda = (t, \ldots, t)$ with the $t$ repeated $n-1$ times. Observe that $\lambda' = (n-1, \ldots, n-1)$ where the term $n-1$ is repeated $t$ times. Then apply Theorem \ref{mu_subgps}.
\end{proof}

\begin{prop}\label{deg_mu_subgps}
	The expression in Corollary \ref{mu_subgps_cor} is a polynomial in $p$ of degree
	$$
	\sum_{j=1}^t \nu_{j}' (n-1-\nu_j').
	$$
\end{prop}
\begin{proof}
	The degree of $\pbin{n}{k}_p$ is equal to $k(n-k)$. Therefore the degree of each term in the product is
	$$
	\nu_{j+1}' (n-1-\nu_j') + (n-1-\nu_{j+1}' - (\nu_j' -\nu_{j+1}'))(\nu_j' - \nu_{j+1}').
	$$
	Simplifying and then adding all terms together gives the result.
\end{proof}

\begin{thm}\label{subgroup_lb}
	Fix $n>1$ and let $\lambda = (t, \ldots, t)$ be the type of the group $G = \left(\Z/p^t\Z\right)^{n-1}$. Let $k \in [0,t(n-1)]$. Set $b = \floor{\frac{k}{t}}$ and $c = \ceil{\frac{k}{t}}$. Then the number of subgroups of $G$ of order $p^k$ is a polynomial in $p$ with degree
	$$ 
k(n-1) -c^2(k+t-ct) - b^2(ct-k).
	$$
\end{thm}
\begin{proof}
	By Proposition \ref{deg_mu_subgps}, for each $\nu \subseteq \lambda = (t,t, \ldots, t)$, the degree of the polynomial that counts the number of subgroups of type $\nu$ is 
	$$
	\sum_{j=1}^{t} \nu_{j}' (n-1-\nu_j').
	$$
	The number of subgroups of type $\nu \subseteq \lambda$ is equal to the number of subgroups of type $\nu' \subseteq \lambda'$. 
	Therefore we seek to maximize the degree of the polynomials counting subgroups of order $p^k$ in $\left(\Z/p^t\Z\right)^{n-1}$ by considering the subgroups of type $\nu' \subseteq \lambda' = (n-1, \ldots, n-1)$ where the $n-1$ is repeated $t$ times subject to the constraint $\sum_{j=1}^{t} \nu_j' =k$. Observe that the degree is as large as possible when 
	$
	\sum_{j=1}^t (\nu_j')^2
	$
	is minimized subject to the constraint that $\sum_{j=1}^t \nu_j' = k$. Over $\R$, this function is minimized when $\nu_j' = \frac{k}{t}$ for all $j = 1, \ldots t$. However, each $\nu_j'$ must be an integer. Thus take $i$ of the $\nu_j' = \floor{\frac{k}{t}}$ and the rest equal to $\ceil{\frac{k}{t}}$ subject to the constraint $i \floor{\frac{k}{t}} + (t-i) \ceil{\frac{k}{t}} =k$. Then 
	$$
	\sum_{j=1}^t (\nu_j')^2 = i \left\lfloor\frac{k}{t} \right\rfloor^2+ (t-i) \left\lceil\frac{k}{t}\right\rceil^2.
	$$
	Solving the constraint for $i$ gives $i = t \ceil{\frac{k}{t}} - k$. The result follows.
\end{proof}

Theorem \ref{subgroup_lb} gives the degree of the polynomial that counts the number of subgroups of order $p^k$ in $\left(\Z/p^t\Z\right)^{n-1}$ for each $k\in [0,t(n-1)]$. Applying Theorem \ref{subgps_to_subrings} and Theorem \ref{subgroup_lb} gives a lower bound for the number of subrings of index $p^{t(n-1)+k}$ in $\Z^n$ for each $k\in [0, t(n-1)]$. We can now state a lower bound for subrings of index $p^e$ based on these results. Moreover, our result holds for the function that counts subrings in $\OO_K$ of index $p^e$; however for consistency throughout this paper, we state the result for the function $f_n(p^e)$, which counts subrings in $\Z^n$ of index $p^e$.

\begin{cor}\label{subgroup_cor}
	Fix integers $n> 1$ and $e \ge n-1$. Fix an integer $t$ so that $t(n-1) \le e \le 2t(n-1)$. Let $k = e -t(n-1)$. Set $b = \floor{\frac{e}{t}} - (n-1)$ and $c = \ceil{\frac{e}{t}} - (n-1)$. Then 
	$$ 
f_n(p^e) \ge p^{k(n-1) -c^2(k+t-ct) - b^2(ct-k)}.
	$$
\end{cor}

\begin{rmk}\label{subgrp_rmk}
	The bound in Corollary \ref{subgroup_cor} is as large as possible when $t \mid e$. Observe that the exponent of the bound reduces in this case. We find that if $t(n-1) \le e  \le 2t(n-1)$ and $t \mid e$, then 
	$$
	f_n(p^e) \ge p^{(e - t(n-1))(2(n-1)-\frac{e}{t})}.
	$$
\end{rmk}

Finally, we prove Theorem \ref{fn_bd_subgps}.

\begin{proof}[Proof of Theorem \ref{fn_bd_subgps}] 
	This theorem follows from Corollary \ref{subgroup_cor}, taking a maximum over all $t$ so that $t(n-1) \le e \le 2t(n-1)$.
\end{proof}


\section{A method for counting subring matrices} \label{method_counting}
	We now introduce a combinatorial method for counting subrings due to Liu \cite{liu}. In Section \ref{fn_bound_sect}, we will use this method to find a different bound for subrings of prime power index in $\Z^n$. In this section, we begin by describing Liu's method for counting subrings and we then describe a simpler algorithm for counting subrings based on row reduction.
	
	\begin{defn}
		A $n \times n$ matrix with entries in $\mb{Z}$ is in \emph{Hermite normal form} if $A = (a_{ij})_{i,j}$ is upper triangular and $0 \leq a_{ij} < a_{ii}$ for all $1 \leq i < j \leq n$. 
	\end{defn}
	
	\begin{defn}
	 An $n \times n$ matrix $A$ in Hermite normal form is a \emph{subring matrix} of index $k$ if
	\begin{enumerate}
		\item $\text{det}(A) = k$.
		\item The identity $(1, 1, \ldots, 1)^T$ is in the column span of $A$.
		\item For all $1 \leq i \leq j \leq n$, if $v_i = (v_1, v_2, \ldots, v_n)^T$ and $v_j = (w_1, w_2, \ldots, w_n)^T$ are columns of $A$, then $v_i \circ v_j =( v_1w_1 , v_2 w_2, \ldots ,v_nw_n)^T$ is in the column span of $A$.
	\end{enumerate}
\end{defn}

 A subring matrix $A$ is \emph{irreducible} if $\text{det}(A) = p^e$, every element in the first $n-1$ columns is divisible by $p$, and the last column is $(1, 1, \ldots, 1)^T$. A subring in $\Z^n$ of index $p^e$ is \emph{irreducible} if for every $(v_1, \ldots, v_n) \in \Z^n$, $v_1 \equiv v_2 \equiv \cdots \equiv v_n \pmod{p}$.
 
 Liu justifies the terminology ``irreducible subring" by showing that any subring $S$ of index $p^e$ in $\Z^n$ can be written uniquely as a direct sum of irreducible subrings $S_i$ in $\Z^{n_i}$, see \cite[Theorem 3.4]{liu}.

	In \cite{liu}, Liu shows that there is a bijection between subrings in $\mb{Z}^n$ of index $k$ and subring matrices with determinant $k$. He also shows that there is a bijection between irreducible subrings of $\Z^n$ of index $p^e$ and irreducible subring matrices with determinant $p^e$.
	
 Let $g_n(k)$ be the number of irreducible subrings of index $k$ in $\Z^n$. Observe that $g_n(k)$ is denoted $g_{n+1}(k)$ in \cite{liu}.

\begin{prop}\cite[Proposition 4.4]{liu} \label{recurrence} There is a recurrence relation
$$	f_n(p^e)  = \sum_{i=0}^e \sum_{j=1}^{n} \binom{n-1}{j-1} f_{n-j}(p^{e-i}) g_j(p^i).$$
\end{prop}

	By the above recurrence relation, to understand $f_n(p^e)$, it suffices to understand $g_j(p^i)$ for each prime $p$, $j \leq n$ and $i \leq e$. In particular, we study the number of irreducible subring matrices of index $p^e$. These matrices have the following form
	$$
	A = \begin{pmatrix}
	p^{e_1} & pa_{12} & pa_{13} & \cdots & pa_{1(n-1)}&1 \\
	& p^{e_2} & pa_{23} & \cdots &pa_{2(n-1)} &1\\
	&& p^{e_3} & \cdots & pa_{3(n-1)} &1\\
	&&& \ddots & \vdots& \vdots\\
	&&&&p^{e_{n-1}} &1\\
	&&&&&1
	\end{pmatrix}
	$$
	where $0 \leq a_{ij} < p^{e_i-1}$ for each $1\leq i < j \leq n-1$. 
	\begin{defn}
	We say that the matrix $A$ as written above has \textit{diagonal} $(p^{e_1}, \ldots, p^{e_{n-1}}, 1)$.
	\end{defn}

Let $C_{n,e}$ denote the number of compositions of $e$ into $n-1$ parts. The diagonals corresponding to irreducible subring matrices with determinant $p^e$ are in bijection with the compositions in $C_{n,e}$. If a matrix has diagonal $(p^{e_1}, \ldots, p^{e_{n-1}}, 1)$, abusing notation we also say the matrix has diagonal $\alpha =(e_1, \ldots, e_{n-1}) \in C_{n,e}$. Let $g_\alpha(p)$ be the number of irreducible matrices with diagonal $\alpha \in C_{n,e}$. Then
	$$
	g_n(p^e) = \sum_{\alpha \in C_{n,e}} g_\alpha(p).
	$$
	
	The authors of \cite{liu} and \cite{akkm} prove that $f_n(p^e)$ is polynomial in $p$ when $e \leq 8$ and $n  > 0$ by providing exact formulas for the number of irreducible subrings with each possible diagonal and then using the recurrence relation from Proposition \ref{recurrence}. 
	
	Let $\text{Col}(A)$ denote the $\Z$-column span of the matrix $A$. For each $\alpha \in C_{n,e}$, the authors of \cite{liu} and \cite{akkm} determine formulas for $g_\alpha(p)$ by considering all closure conditions $v_i \circ v_j \in \text{Col}(A)$ explicitly in order to find conditions on the variables $a_{ij}$.  In general, finding all the closure conditions for a given diagonal can be complicated.
	
	\begin{example}\cite[Page 19]{akkm}
		Let $\alpha = (3, 2, 1, 1)$. The corresponding matrix is
		$$
		\begin{pmatrix}
		p^3 & pa_{12} & pa_{13} & pa_{14} & 1\\
		& p^2  & pa_{23}&pa_{24} & 1\\
		&& p &0 & 1\\
		&&&p&1\\
		&&&&1
		\end{pmatrix}
		$$
		with $a_{1j} \in [0,p^2)$ for $j=2,3,4$ and $a_{2j} \in [0,p)$ for $j =3,4$. Atanasov et. al illustrate their method for finding the closure conditions.  For example, consider $v_2 \circ v_2 = (p^2a_{12}^2, p^4 , 0 ,0 ,0)^T$. They note that $v_2 \circ v_2$ must be a linear combination of the first two columns, so it suffices to understand how to write $(p^2a_{12}^2, p^4)^T$ as a linear combination of $(p^3 , 0)^T$ and $(pa_{12}, p^2)^T$. We must take $p^2$ times this second column, so we obtain
		\begin{equation}\label{closure_cond_ex}
		\begin{pmatrix}
		p^2a_{12}^2\\p^4 
		\end{pmatrix} = \lambda\begin{pmatrix}
		p^3\\0
		\end{pmatrix}
		+ 
		p^2\begin{pmatrix}
		pa_{12}\\p^2
		\end{pmatrix}
		\end{equation}
for some $\lambda \in \Z$. In order for Equation \ref{closure_cond_ex} to hold, $p^3 \mid p^2a_{12}$ and thus $p\mid a_{12}$. Thus the closure condition corresponding to $v_2 \circ v_2$ is $a_{12} \equiv 0 \pmod{p}$. Atanasov et al. compute the closure conditions for other pairs of columns using similar methods and noting that the equations are simpler when they replace $a_{12}$ with $pa_{12}'$. The final conditions are 
	\begin{align*}
		&(a_{13}^2 - a_{13}) - (a_{23}^2-a_{23})a_{12}' \equiv 0 \pmod{p}\\
		&(a_{14}^2 - a_{14}) - (a_{24}^2-a_{24})a_{12}' \equiv 0 \pmod{p}\\
		& a_{13}a_{14} -a_{23}a_{24}a_{12}'\equiv 0 \pmod{p}.
	\end{align*}
	\end{example}

	 We show that by using the method of row reduction, we can find these conditions in an algorithmic way. This method has two advantages. First, it simplifies proofs about counting the number of irreducible subrings with a given diagonal. Second, since it is algorithmic, it is easily implementable in a computer algebra system.
	
	Recall that the closure conditions are of the form $v_i \circ v_j\in \text{Col}(A)$ for each $1 \leq i\leq j\leq n$. These conditions are satisfied if and only if $A \vec{x}_{i, j} = v_i \circ v_j$ has a solution $\vec{x}_{i, j} \in \mb{Z}^n$ for every $1\leq i \leq j \leq n$. By basic linear algebra, a solution $\vec{x}_{i,j}$ exists if and only if the last column in the reduced echelon form of the matrix $[A \; v_i \circ v_j]$ has integer entries. We illustrate the row reducing steps below. Note that we omit the column $(1, 1, \ldots, 1)^T$ corresponding to the identity in $\Z^n$ since it is clear that $v_i \circ (1, 1, \ldots, 1)^T$ is in the $\Z$-column span of $A$ for all $1 \leq i \leq n$ and if $i,j \ne n$, then the $n^{th}$ entry in $v_i \circ v_j$ is 0.	Suppose we have a matrix
	$$
	A = \begin{pmatrix}
	p^{e_1} & pa_{12} & pa_{13} & \cdots & pa_{1(n-1)} & x_1\\
	& p^{e_2} & pa_{23} & \cdots &pa_{2(n-1)} & x_2\\
	&& p^{e_3} & \cdots & pa_{3(n-1)} & x_3\\
	&&& \ddots & & \vdots\\
	&&&&p^{e_{n-1}} & x_{n-1}
	\end{pmatrix}
	$$

	\noindent where the last column contains the first $n-1$ entries of the vector $v_i \circ v_j = (x_1, x_2, \ldots, x_{n-1}, x_n)^T$ for some pair $(i,j)$ so that $1 \leq i \leq j \leq n-1$.
	We begin by dividing each row $i$ by $p^{e_i}$ in order to make every diagonal entry equal to 1, obtaining the matrix
	
	$$
	A \rightarrow
	\begin{pmatrix}
	1 & \frac{pa_{12}}{p^{e_1}} & \frac{pa_{13}}{p^{e_1}} & \cdots & \frac{pa_{1(n-1)}}{p^{e_1}} & \frac{x_1}{p^{e_1}}\\[6pt]
	& 1 & \frac{pa_{23}}{p^{e_2}} & \cdots &\frac{pa_{2(n-1)}}{p^{e_2}} &\frac{x_2}{p^{e_2}}\\[6pt]
	&& 1 & \cdots & \frac{pa_{3(n-1)}}{p^{e_3}} & \frac{x_3}{p^{e_3}}\\[6pt]
	&&& \ddots & & \vdots\\[6pt]
	&&&&1 & \frac{x_{n-1}}{p^{e_{n-1}}}
	\end{pmatrix}.
	$$
	
	Observe that since the matrix is upper triangular, it is not necessary to make any more divisions. From here, we row reduce and we note that finding conditions on $a_{ij}$ so that $v_i \circ v_j$ is in the $\Z$-span of the first $n-1$ columns is equivalent to finding conditions on $a_{ij}$ so that the entries in the last column are in $\Z$. Thus, for each fixed $v_i \circ v_j$, we are counting the number of solutions to expressions of the form
	\begin{equation}
	\label{conditions}
	h_{ij}^{(c)} = \frac{f_{ij}^{(c)}(\vec{a_{ij}})}{p^{r}} \in \mb{Z}
	\end{equation}
	where $f_{ij}^{(c)}$ is a multivariate polynomial with coefficients in $\Z$, $1 \leq c \leq n-1$, and $r$ is an integer depending on $i, j,$ and $c$. Note that for this fixed $v_i \circ v_j$, $x_{i+1} = x_{i+2} = \cdots = x_{n-1} = 0$. This allows us to use the row reduction process on the smaller $i \times (i+1)$ matrix formed by taking the first $i$ columns of $A$, augmenting $v_i \circ v_j$, and then removing the rows $i+1, \ldots, n-1$. We will make use of this simplification in Section \ref{fn_bound_sect}.
	
	Let $p^r$ be the largest denominator that occurs in all expressions $h_{ij}^{(c)}$ in Equation \ref{conditions}. Then the rational functions $h_{ij}^{(c)}$ are in $\mb{Z}$ if and only if the numerators $f_{ij}^{(c)}\equiv 0 \mod{p^r}$. Thus counting the simultaneous system $h_{ij}^{(c)} \in \mb{Z}$ is essentially the same as counting points to a simultaneous vanishing of the polynomials $f_{ij}^{(c)}$ modulo $p^r$. Note that these problems are not exactly the same as the variables $a_{ij}$ live in the range $[0, p^{e_i-1})$ and we may have $e_i -1 > r$. However, since all expressions have denominator at worst $p^r$, then the congruence conditions for the variables only matter modulo $p^r$. For any $e_i -1 > r$, we can simply multiply the point count for the variety by $p^{e_i-1-r}$. Thus up to polynomial factors, it is sufficient to understand the varieties defined over $\mb{Z}/ p^r\mb{Z}$ defined by the polynomials $f_{ij}^{(c)}$.
	
	\begin{rmk}
		In general, counting points on varieties over a ring is a difficult problem. It is easier to consider varieties over $\mb{F}_p$, however these denominators that occur are often larger than $p$. Sometimes it is possible to reduce the denominators to all be $p$ by using clever substitution. In these cases, we can say more about the point counts of the varieties.
	\end{rmk}

\noindent We summarize the method described in this section as follows.
\begin{alg}
	Input a diagonal $(e_1, \ldots, e_{n-1})$ with integers $e_i \ge 0$ for all $1 \le i \le n-1$.
	\begin{enumerate}
\item Create the matrix
	$$
	A = \begin{pmatrix}
	p^{e_1} & pa_{12} & pa_{13} & \cdots & pa_{1(n-1)}&1 \\
	& p^{e_2} & pa_{23} & \cdots &pa_{2(n-1)} &1\\
	&& p^{e_3} & \cdots & pa_{3(n-1)} &1\\
	&&& \ddots & \vdots& \vdots\\
	&&&&p^{e_{n-1}} &1\\
	&&&&&1
\end{pmatrix}
	$$
	in the variables $a_{ij}$ for $1 \le i < j \le n-1$. If $e_i =0$ or 1, set $a_{ij} =0$ for all $i < j \le n-1$.
\item For each $v_i \circ v_j$ with $1 \le i \le j \le n-1$, row reduce $[A \, v_i \circ v_j]$ to $A'$ over $\Q$. Add the entries of the rightmost column of $A'$ to a list.
\item Return the list formed in Step 2. All elements in this list are of the form $\frac{f_{ij}(\vec{a}_{ij})}{p^r}$ for some $r \ge 0$.
\end{enumerate}
\end{alg}

\section{A new lower bound for $f_n(p^e)$ via irreducible subring matrices}\label{fn_bound_sect}
 We now provide a lower bound for the number of subrings in $\Z^n$ of index $p^e$ using techniques from Section \ref{method_counting}. That is, we find a lower bound for $f_n(p^e)$ by bounding the number of \emph{irreducible} subrings of index $p^e$ in $\Z^n$. These results will lead to a new lower bound for the abscissa of convergence of $\zeta_{\Z^n , \, p}^R(s)$, which will be discussed in Section \ref{asymptotic_section}.

\subsection{Bounding the number of irreducible subring matrices}

Fix an integer $d \in [0,n-1]$ and let $k, \ell$ be positive integers so that $\ell \ge \ceil{\frac{k}{2}}$. Let $C_{n, d, k, \ell}$ denote the set of compositions of $kd + \ell(n-1-d)$ into $n-1$ parts that contain exactly $d$ terms equal to $k$ and $n-1-d$ terms equal to $\ell$.  In other words, each $\alpha \in C_{n,d,k,\ell}$ is a permutation of the composition $(k, k, \ldots, k, \ell, \ell, \ldots, \ell)$ with $k$ appearing $d$ times and $\ell $ appearing $n-1-d$ times.
\begin{prop}\label{21_prop}
Let $n >1$. Fix an integer $d \in [0,n-1]$ and let $k, \ell$ be positive integers so that $\ell \ge \ceil{\frac{k}{2}}$. For a fixed $\alpha \in C_{n,d,k,\ell}$, let $A_\alpha$ be a matrix in Hermite normal form with diagonal $\alpha$ that satisfies the following conditions for each pair $1 \le i< j \le n-1$:
\begin{enumerate}
	\item  if $a_{ii} = p^k$ and $a_{jj} = p^\ell$, then $a_{ij} \equiv 0 \pmod{p^{\ceil{\frac{k}{2}}}}$
	\item otherwise, $a_{ij} =0$.
\end{enumerate}
Then $A_\alpha$ is an irreducible subring matrix.
\end{prop}
\begin{proof}
	Let $A_\alpha$ be as described in the statement of the proposition. We must show that $v_i \circ v_j \in \text{Col}(A_\alpha)$ for all $ 1\leq i \leq j \leq n$.  Let $r_1, \ldots, r_d$ be the columns containing $p^k$. There are three cases to consider. 
	
	 First, suppose that $i=r_m$ for some integer $m \in [1,d]$. Then $v_i = (0 ,\cdots, 0 ,p^k , 0 ,\cdots ,0)^T$. Fix some integer $j \in [1,n]$. Suppose the $i^{th}$ entry of $v_j$ is equal to $x$. Then $v_i \circ v_j = x v_i \in \text{Col}(A_\alpha)$. Notice that we made no assumption about the entries of $v_j$. 
	
	Second, suppose the $i^{th}$ entry of $v_i$ is $p^\ell$ and let $j > i$ such that $j \not \in \{r_1, \ldots, r_d\}$. The only possible nonzero entries in $v_i \circ v_j$ are in the rows $\{i_1, \ldots, i_d\}$. Therefore $v_i \circ v_j \in \text{Col}(A_\alpha)$ if and only if $v_i \circ v_j$ is a linear combination of $v_{r_1}, \ldots, v_{r_d}$. By applying the technique established in Section \ref{method_counting}, we need to understand whether the following augmented matrix has integer solutions
	$$A_\alpha' = 
\begin{pmatrix}
	p^k& 0 & \cdots & 0 &  p^{2\ceil{\frac{k}{2}}} a_{r_1i}a_{r_1j}\\
	& p^k & \cdots & 0  & p^{2\ceil{\frac{k}{2}}}a_{r_2i}a_{r_2j}\\
	&& \ddots & \vdots & \vdots\\
	&&&p^k  & p^{2\ceil{\frac{k}{2}}}a_{r_di}a_{r_dj}
\end{pmatrix}.
$$
Observe that $v_i \circ v_j \in \text{Col}(A_\alpha)$ if and only if $$v=(p^{2\ceil{\frac{k}{2}}}a_{r_1i}a_{r_1j}\,,\, p^{2\ceil{\frac{k}{2}}}a_{r_2i}a_{r_2j} \, ,\, \cdots\, ,\, p^{2\ceil{\frac{k}{2}}}a_{r_di}a_{r_dj})^T\in \text{Col}(A_\alpha').$$
By applying the row reduction technique, we see that $v \in \text{Col}(A_\alpha')$ if and only if 
$$
p^k \mid  p^{2\ceil{\frac{k}{2}}} a_{r_mi}a_{r_mj}
$$
for each integer $m \in [1,d]$. This condition holds for all $m \in [1,d]$.
	

	Finally, suppose that $v_i$ has a $p^\ell$ in the $i^{th}$ entry and consider $v_i \circ v_i$. Then $v_i \circ v_i \in \text{Col}(A_\alpha)$ if and only if $v_i \circ v_i$ is a linear combination of $v_i$ and $v_{r_1}, \ldots, v_{r_d}$. Consider the matrix
	$$A_\alpha' = 
	\begin{pmatrix}
	p^k& 0 & \cdots & 0 & p^{\ceil{\frac{k}{2}}} a_{r_1 i} & p^{2\ceil{\frac{k}{2}}} a_{r_1i}^2\\
	& p^k & \cdots & 0 & p^{\ceil{\frac{k}{2}}} a_{r_2i} & p^{2\ceil{\frac{k}{2}}}a_{r_2i}^2\\
	&& \ddots & \vdots & \vdots & \vdots\\
	&&&p^k &p^{\ceil{\frac{k}{2}}}a_{r_di} & p^{2\ceil{\frac{k}{2}}}a_{r_di}^2\\
	&&&&p^\ell& p^{2\ell}
	\end{pmatrix}.
	$$
 After applying the row reduction method, $v_i \circ v_i \in \text{Col}(A_\alpha)$ if and only if $$p^k \mid \left(p^{2\ceil{\frac{k}{2}}} a_{r_mi}^2 - p^{\ell+\ceil{\frac{k}{2}}}a_{r_mi}\right)$$ for all integers $m \in [1,d]$. This holds for all possible choices of $a_{r_mi}$ since $\ell \ge \ceil{\frac{k}{2}}$.
\end{proof}
 
\begin{example}\label{ex_alpha_1}
	Let $\alpha = (2, 1, 2, 1, 2)$. The matrix $A_\alpha$ corresponding to Proposition \ref{21_prop} has the form
	$$
	\begin{pmatrix}
	p^2 & pa_{12} & 0 & pa_{14} & 0 & 1\\
	&p & 0 & 0 & 0 &1\\
	&&p^2&pa_{34} & 0& 1\\
	&&&p & 0 &1\\
	&&&&p^2&1\\
	&&&&&1
	\end{pmatrix}.
	$$
	There are exactly $p^3$ irreducible subring matrices of this form since $A_\alpha$ is a subring matrix for any choice of $a_{12}, a_{14}, a_{34} \in [0,p)$.
\end{example}

\begin{example} \label{ex_alpha_2}
	Let $\alpha = (3, 5, 3, 3, 5)$. Setting $k=5$, we have $\ceil{\frac{k}{2}} = 3$. The matrix $A_\alpha$ corresponding to Proposition \ref{21_prop} has the form
	$$
	\begin{pmatrix}
	p^3 & 0 & 0 & 0 & 0 & 1\\
	&p^5 & p^3a_{23} & p^3a_{24} &0 &1\\
	&&p^3&0 & 0& 1\\
	&&&p^3 & 0&1\\
	&&&&p^5&1\\
	&&&&&1
	\end{pmatrix}.
	$$
	When $i =2$, $a_{ij} \in [0, p^2)$. There are exactly $p^4$ such irreducible subrings matrices.
\end{example}

We now discuss a method for computing the number of subring matrices that have the form given in Proposition \ref{21_prop}.

\begin{defn}
	A \emph{north-east lattice path} $P$ is a path in $\Z^2$ starting at the origin and ending at $(u,v)$ so that every step in the path is either a step one unit to the north or one unit to the east. 
	
	The \emph{area} of a path is the area enclosed by the path, the $x$- and $y$-axes, and the line $x =u$. Denote the area by $\text{Area}(P)$.
\end{defn}

Fix $\alpha \in C_{n,d,k,\ell}$ and let $A_{\alpha}$ be a matrix as in Proposition $\ref{21_prop}$. Let $P_\alpha$ denote the lattice path from $(0,0)$ to $(n-1-d,d)$ so that the $i^{th}$ step in the path is a northerly step if the $i^{th}$ entry of $\alpha$ is $k$ and is an easterly step if the $i^{th}$ entry of $\alpha$ is $\ell$.

\begin{thm}\label{21_num}
	Let $\alpha \in C_{n, d, k, \ell}$ and let $A_\alpha$ be as in Proposition \ref{21_prop}. Then the number of such matrices $A_\alpha$ is equal to $p^{(k - \ceil{\frac{k}{2}})\cdot\text{Area}(P_\alpha)}$. 
\end{thm}
\begin{proof}
	Observe that a non-diagonal element $a_{ij} \in A_\alpha$ is nonzero if and only if $i < j$ and $\alpha_i = k, \alpha_j = \ell$. In this case any choice of $a_{ij} \in [0, p^{k - \ceil{\frac{k}{2}}})$ leads to an irreducible subring matrix. By the definition of $P_\alpha$, we see that
	$$
	\text{Area}(P_\alpha) = \#\{i < j\, : \, \alpha_i = k \text{ and } \alpha_j = \ell\}.
	$$
	Therefore the number of irreducible subring matrices $A_\alpha$ satisfying the conditions in Proposition \ref{21_prop} is 
	$$\left(p^{(k - \ceil{\frac{k}{2}})}\right)^{\text{Area}(P_\alpha)}.$$

\end{proof}

\begin{cor}
	For each $\alpha \in C_{n, d, k, \ell}$, we have $g_\alpha(p) \ge p^{(k - \ceil{\frac{k}{2}})\cdot\text{Area}(P_\alpha)}$.
\end{cor}

\begin{example}
	Let $\alpha = (2, 1, 2, 1, 2)$. The north-east lattice path $P_\alpha$ goes from $(0,0)$ to $(2,3)$, following the steps: north, east, north, east, north. This path is depicted below.
	
	
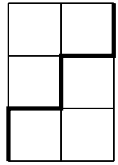
\begin{center}
		\begin{tikzpicture}[scale=0.7]
		\foreach \i in {0,...,2}{
			\foreach \j in {0,...,3}{
				\draw (\i,0)--(\i,3);
				\draw (0,\j)--(2,\j);
			}
		}
		\draw[ultra thick] (0,0)--(0,1)--(1,1)--(1,2)--(2,2)--(2,3);
		\end{tikzpicture}
	\end{center}
	
	\noindent The area of this path is equal to 3 and $k - \ceil{\frac{k}{2}} = 1$, verifying our claim in Example \ref{ex_alpha_1} that there are exactly $p^3$ irreducible subring matrices that have diagonal $\alpha$ and satisfy the properties listed in Proposition \ref{21_prop}.
\end{example}

\begin{rmk} \label{main_term_rmk}
For each $\alpha \in C_{n, d, k, \ell}$ let $A_\alpha$ be a matrix satisfying Proposition \ref{21_prop}. Let $P_\alpha$ be the corresponding north-east lattice path. Set $\gamma = (k, k, \ldots, k, \ell, \ldots, \ell)\in C_{n,d,k,\ell}$. Observe that $\text{Area}(P_\gamma) \geq \text{Area}(P_\alpha)$ for all $\alpha \in C_{n, d, k, \ell}$. Therefore the degree of the main term in our bound is always equal $(k - \ceil{\frac{k}{2}}) \text{Area}(P_\gamma)$.
\end{rmk}

\begin{cor} \label{fn_bound_cor}
	Set $\gamma = (k, \ldots, k, \ell, \ldots, \ell)$. Then $g_\gamma(p) \ge p^{(k - \ceil{\frac{k}{2}}) d(n-1-d)}$.
\end{cor}
\begin{proof}
	The north-east lattice path $P_\gamma$ is a rectangle with vertices $(0,0), (n-1-d,0), (n-d-1,d), $ and $(0,d)$. This rectangle has area $d(n-1-d)$. The result follows from Theorem \ref{21_num}.
\end{proof}

To conclude this section, we show that the method of counting subrings recovers the bound for $f_n(p^{n-1+d})$ when $d \in [0,n-1]$ given in Proposition \ref{brakenhoff_refined_bound}, thus giving a different proof of this proposition.
\begin{lemma}\cite[Page 116]{cameron}\label{cameron_lem}
	Let $\mathcal{P}$ be the set of north-east lattice paths from $(0,0)$ to $(u,v)$ and let $q$ be a prime power. Then
	$$
	\sum_{P\in \mathcal{P}} q^{\text{Area}(P)} = \pbin{u+v}{v}_q.
	$$
\end{lemma}

\begin{cor}
	Fix integers $n > 1$ and $d\in[0,n-1]$. Then $g_n(p^{n-1+d}) \geq \pbin{n-1}{d}_p.$
\end{cor}
\begin{proof}
	Recall that $g_n(p^e) = \sum_{\alpha \in C_{n,e}} g_\alpha(p)$. We can bound $g_n(p^{n-1+d})$ from below by counting the number of irreducible subrings with diagonal $\alpha \in C_{n, d, 2, 1}$. The set $ C_{n, d, 2, 1}$ is in bijection with the set of north-east lattice paths from $(0,0)$ to $(d, n-1-d)$. Fix some $\alpha\in  C_{n, d, 2, 1}$. By Theorem \ref{21_num}, for the corresponding lattice path $P_\alpha$, $g_\alpha(p) \geq p^{\text{Area}(P_\alpha)}$. By Lemma \ref{cameron_lem},
	$$
	g_n(p^{n-1+d}) \geq \sum_{\alpha \in  C_{n, d, 2, 1}} g_\alpha(p) \ge \pbin{n-1}{d}_p.
	$$
\end{proof}
Since $f_n(p^{e}) \geq g_n(p^{e})$ for all $e \geq 0$, $f_n(p^{n-1+d}) \geq  \pbin{n-1}{d}_p.$ Thus this method of bounding subrings in $\Z^n$ gives exactly the same bound as Proposition \ref{brakenhoff_refined_bound}. 

\subsection{Optimizing the bound for $g_n(p^e)$} \label{optimized_bd}

In this section, we optimize the exponent $(k-\ceil{\frac{k}{2}})d(n-1-d)$. Fixing $n$, Corollary \ref{fn_bound_cor} implies that 
$$f_n(p^e) \geq g_n(p^e) \geq p^{(k-\ceil{\frac{k}{2}})d(n-1-d)}$$
 for each $0 \leq d \leq n-1$ and $k, \ell \in \Z_{\geq 1}$ so that $\ell \geq \frac{k}{2}$ and $e \le kd+ \ell(n-1-d)$. It is important that $k, \ell \geq 1$; if $e < n-1$, then $g_n(p^e)=0$. In order to obtain the best possible bound for $f_n(p^e)$ in terms of $n$ and $e$, we optimize the exponent $(k-\ceil{\frac{k}{2}})d(n-1-d)$ over $\Z$. 


\begin{proof}[Proof of Theorem \ref{fn_bound_1}]
The term $k-\ceil{\frac{k}{2}}$ is maximized when $k = 2j$ for some $j \in \N$. Recall that we are subject to the constraint $kd + \ell(n-1-d) \geq e$ for some $\ell \ge \ceil{\frac{k}{2}}$. Therefore $j \leq \floor{\frac{e}{d+(n-1)}}.$ Set $j = \floor{\frac{e}{d+(n-1)}}$ so that it is as large as possible.

Then $(k-\ceil{\frac{k}{2}}) d(n-1-d) \geq  \floor{\frac{e}{d+(n-1)}} \cdot d (n-1-d).$ Taking a maximum over all $0 \leq d \leq n-1$ gives the result.
\end{proof}
The above proposition gives the best possible bound for $(k-\ceil{\frac{k}{2}})d(n-1-d)$ subject to the constraints that $\ell \geq \ceil{\frac{k}{2}}$ and $kd + \ell(n-1-d) \geq e$. We now give a weakening of Theorem \ref{fn_bound_1}, which will be helpful later. The benefit of the following proposition is that the maximum is taken over real numbers rather than integers.

\begin{prop}  \label{fn_bound_lemma}
	Suppose that $e \ge n-1$. Let $$c(n,e) = \max_{0 \leq C \leq 1}\left(e \left( \frac{C - C^2}{C+1} (n-1) + \frac{C-1}{C+1}\right) - \left( (C-C^2)(n-1)^2 +(C-1)(n-1)\right)\right)$$ where the maximum is taken over $\R$. Then 
	$f_n(p^e) \geq p^{c(n,e)}.$
\end{prop}

\begin{proof}
For any fixed integer $d \in [0,n-1]$, $d = \floor{C(n-1)}$ for some real number $C \in [0,1]$. Starting from the bound given in Theorem \ref{fn_bound_1}, for any $C \in [0,1]$,
\begin{align*}
\left(k-\ceil{\frac{k}{2}}\right)d\left(n-1-d\right) &\ge \bigg\lfloor\frac{e}{\lfloor C(n-1) \rfloor + (n-1)} \bigg\rfloor \cdot \lfloor C(n-1) \rfloor \cdot \left(n-1- \lfloor C(n-1) \rfloor\right)\\
&\geq  \left(\frac{e}{\lfloor C(n-1) \rfloor + (n-1)}-1 \right)(C(n-1)-1)(1-C)(n-1)\\
&\geq \left(\frac{e}{(C+1) (n-1)}-1 \right)(C(n-1)-1) (1-C)(n-1)\\
&= \left(\frac{e}{(C+1) (n-1)}-1 \right) \left( (C-C^2)(n-1)^2 +(C-1)(n-1)\right) \\
&= e \left( \frac{C - C^2}{C+1} (n-1) + \frac{C-1}{C+1}\right) - \left( (C-C^2)(n-1)^2 +(C-1)(n-1)\right).
\end{align*}
Taking a maximum over all real numbers $C \in [0,1]$ gives the result.
\end{proof}

\subsection{Comparison of Theorems \ref{fn_bd_subgps} and \ref{fn_bound_1}}
Consider the bounds from Theorems \ref{fn_bd_subgps} and \ref{fn_bound_1}. We compared these lower bounds for various values of $n$ and $e$ in Sage and found that they grow at very similar rates. The bound from Theorem \ref{fn_bd_subgps} seems to be slightly better than the bound from Theorem \ref{fn_bound_1} for each fixed $n$ and for sufficiently large $e$. Let $h(n,e)$ be the bound from Theorem \ref{fn_bd_subgps} and let $b(n,e)$ be the bound from Theorem \ref{fn_bound_1}. We provide a table summarizing some of this data below.

\begin{center}
	\begin{table}[!htp]\caption{Values of the bounds from Theorems \ref{fn_bd_subgps} and \ref{fn_bound_1}}
\begin{tabular}{|c|c|c|c|}
	\hline
	$n$  & $e$ & $\log_ph(n,e)$ & $\log_p b(n,e)$\\
	\hline
	6 & 10 & 0 & 6\\
	6 & 20 & 16 & 12 \\
	6 & 30 & 24 & 24\\
	6 &300  &256 & 252 \\
	6 & 1000 &  856&852\\
	10 & 10 & 8 & 8\\
	10 & 20 & 16 & 20\\
	10 & 30 & 36 & 40\\
	10 & 300 &460 & 460\\
	10 & 1000 & 1538 & 1520\\
	\hline
\end{tabular}\end{table}
\end{center}
\vspace{8in}
Next, we show that there are subrings that are counted using one of the two techniques, but not the other.

\begin{example}
	Let $n = 3$ and $e = 7$. Consider the following matrix.
	$$
	A = 
	\begin{pmatrix}
	p^3 & p^2 & 1\\
	0 & p^4 & 1\\
	0 & 0 & 1
	\end{pmatrix}
	$$
	This is an irreducible subring matrix with diagonal $\alpha = (3, 4)$. The matrix $A$ satisfies Proposition \ref{21_prop}. Therefore our technique from Section \ref{fn_bound_sect} counts the matrix $A$ in the lower bound for $f_3(p^7)$. 
	
	 Let $G$ be the subgroup generated by the columns of $A$. Then $G$ does not satisfy the condition $\Z+ p^4\Z^3 \subset G \subset \Z+ p^2 \Z^3$ since $(0 , p^4 , 0)^T$ is not in the $\Z$-column span of $A$. Therefore our technique from Section \ref{ext_brak_sec} does not count the subgroup corresponding to $A$.
\end{example}

\begin{example}
	Let $n = 4$. Let $G$ be the subgroup generated by $(1, 1, 1,1), (p^3, 0,0,0)$, $(p^2,p^3,0,0)$, and $(0,0,p^2,0)$. Then $\Z+ p^4 \Z^4 \subset G \subset \Z+ p^2 \Z^4$.
	
	The subgroup $G$ corresponds to the matrix
		$$
	A = 
	\begin{pmatrix}
	p^3 & p^2&0 &1\\
	0 & p^3&0 & 1\\
	0 & 0 &p^2& 1\\
	0&0&0&1
	\end{pmatrix}.
	$$
	Our technique from Section \ref{ext_brak_sec} includes $G$ in the lower bound for $f_4(p^8)$. However, our technique from Section \ref{fn_bound_sect} does not since $A$ violates the conditions in the statement of Proposition \ref{21_prop}.

\end{example}

It is not too difficult to give conditions on when the columns of a subring matrix satisfying Proposition \ref{21_prop} will generate a subgroup $G$ satisfying Lemma \ref{brakenhoff_lemma}. We state this below in Proposition \ref{subring_to_brakenhoff}. However, it is much more difficult to determine when a subgroup $G$ satisfying Lemma \ref{brakenhoff_lemma} corresponds to a subring matrix $M_G$ satisfying Proposition \ref{21_prop}. The main obstacle here is that the closure conditions to determine whether $M_G$ is a subring matrix are complicated.

\begin{prop}\label{subring_to_brakenhoff}
	Let $A$ be a subring matrix satisfying Proposition \ref{21_prop} with columns $v_1, \ldots, v_{n-1}, (1, \ldots, 1)^T$. Let $G$ be the subgroup generated by $v_1^T, \ldots, v_{n-1}^T,$ and $(1, \ldots, 1)$. Then there exists an $r$ such that $\Z+ p^{2r} \Z^n \subset G \subset \Z+ p^r \Z^n$ if and only if 
	$$
	\ell = 
	\begin{cases}
		\frac{k}{2} & 2 \mid k\\[4pt]
		\frac{k+1}{2} \textrm{ or } \frac{k+3}{2} & 2 \nmid k.
	\end{cases}
	$$
\end{prop}
\begin{proof}
	Let $w_1 = (1, 0, \ldots, 0)$, $w_2 = (0, 1, \ldots, 0), \ldots, w_{n-1} = (0, \ldots, 0, 1, 0)$, and $w_n = (1, \ldots, 1)$ be a basis for $\Z^n$. By Proposition \ref{G_lattice_cond}, it suffices to understand conditions on $r, k,$ and $\ell$ so that $p^{2r}w_i$ is contained in the lattice spanned by $G$ and $v_i$ is contained in the lattice $L$ spanned by $\{p^rw_1, \ldots, p^rw_{n-1}\}$ for all $1 \le i \le n-1$. 
	
	Observe that $v_i$ is in $L$ if and only if $$r \le \min\left(\left \lceil\frac{k}{2}\right \rceil, k, \ell\right) = \left \lceil\frac{k}{2}\right \rceil.$$ 
	
	By applying the row reduction method, $p^{2r}w_i$ is in the lattice spanned by $G$ if and only if the following three conditions hold:
	\begin{enumerate}
		\item $k \le 2r$
		\item $\ell \le 2r$
		\item $k \le 2r - \ell  + \ceil{\frac{k}{2}}$.
	\end{enumerate}

We can simplify the four conditions to the following:
\begin{enumerate}
	\item $r = \ceil{\frac{k}{2}}$
	\item $\ceil{\frac{k}{2}} \le \ell  \le 3 \ceil{\frac{k}{2}}-k$.
\end{enumerate}
\end{proof}

Finally, we demonstrate an upper bound for each of our lower bounds, which can be derived as corollaries of Theorems \ref{fn_bd_subgps} and \ref{fn_bound_1} respectively.

\begin{cor} \label{ceil_bd1}
	Let $h(n,e)$ be the exponent of the lower bound from Theorem \ref{fn_bd_subgps}. Then $h(n,e) \le (3- 2\sqrt{2})(n-1)e.$
\end{cor}
\begin{proof}
	 By Remark \ref{subgrp_rmk}, for each fixed $t\in [\frac{e}{2(n-1)}, \frac{e}{n-1}],$ we have $$h(n,e) \le (e-t(n-1))(2(n-1)-\frac{e}{t}).$$ Taking a maximum over all $t$ in this range over $\R$ gives the result.
\end{proof}

\begin{cor}\label{ceil_bd2}
		Let $b(n,e)$ be the exponent of the lower bound from Theorem \ref{fn_bound_1}. Then $b(n,e) \le (3- 2\sqrt{2})(n-1)e.$
\end{cor}
\begin{proof}
	Removing the floor function and optimizing over $\R$ gives the result.
\end{proof}

It is interesting to note that our two different methods lead to lower bounds that are very close asymptotically.

\section{Divergence of local factors} \label{asymptotic_section}
In this section, we use the lower bounds for $f_n(p^e)$ from Section \ref{fn_bound_sect} to find lower bounds for the abscissa of convergence of $\zeta_{\Z^n, \, p}^R(s)$. It is less clear how to use the bound from Theorem \ref{fn_bd_subgps} to derive a result about the divergence of $\zeta_{\Z^n, \, p}^R(s)$ since the exponent is quadratic in $e$. While some results are known about the divergence of $\zeta_{\Z^n}^R(s)$, not much is known about the divergence of the local factors. We fill this gap and also provide a partial answer to Question \ref{conv_ques}.

For each lower bound $p^B$ of $f_n(p^e)$ given in Section \ref{fn_bound_sect}, we can determine the poles of the series $\sum_{e \geq n-1} p^{B} p^{-es}$. Since $f_n(p^e) \geq p^{B}$ whenever $e \geq n-1$, then $\zeta_{\Z^n, \, p}^R(s) \geq \sum_{e \ge n-1} p^Bp^{-es}$. Therefore $\zeta_{\Z^n, \, p}^R(s)$ diverges whenever the simpler series diverges.

First, let $$b(n,e) = \bigg\lfloor\frac{e}{n-1+d} \bigg\rfloor\cdot d (n-1-d)$$ as in Theorem \ref{fn_bound_1}. In order to simplify the geometric series, we set $$F(d,e,n) = \frac{ed(n-1-d)}{n-1+d}  - d(n-1-d)$$ and note that $b(n,e) \geq F(d,e,n)$

\begin{lemma}\label{sum_bound_1} Fix $n >1$. 
	Then $\sum_{e \geq n-1} p^{F(d,e,n)}p^{-es}$ diverges for all $s$ such that $\Re(s) \leq \frac{d(n-1-d)}{n-1+d}$.
\end{lemma}
\begin{proof}
	Consider
	\begin{align*}
	  \sum_{e \ge n-1} p^{F(d,e,n)}p^{-es}&= p^{-d(n-1-d)}\sum_{e \ge n-1} \left(p^{\frac{d(n-1-d)}{n-1+d} - s}\right)^e.
	\end{align*}
	This series diverges for all $s$ such that $\Re(s) \leq \frac{d(n-1-d)}{n-1+d}.$ 
\end{proof}

\begin{lemma}\label{sum_bound_2} Let $n > 1$ and $$G(C, e, n) =  e \left( \frac{C - C^2}{C+1} (n-1) + \frac{C-1}{C+1}\right) - \left( (C-C^2)(n-1)^2 +(C-1)(n-1)\right).$$ Then $ \sum_{e \geq n-1} p^{G(C,e,n)}p^{-es}$ diverges for all $s$ such that $\Re(s) \le \left( \frac{C - C^2}{C+1} (n-1) + \frac{C-1}{C+1}\right)$.
\end{lemma}
\begin{proof} 
		Consider
		\begin{align*}
	\sum_{e \geq n-1} p^{G(C,e,n)} p^{-es}&= p^{- \left( (C-C^2)(n-1)^2 +(C-1)(n-1)\right)}\sum_{e \geq n-1} p^{e \left( \frac{C - C^2}{C+1} (n-1) + \frac{C-1}{C+1}\right)}p^{-es}\\
	  &=p^{- \left( (C-C^2)(n-1)^2 +(C-1)(n-1)\right)} \sum_{ e \geq n-1} \left(p^{ \left(\frac{C - C^2}{C+1} (n-1) + \frac{C-1}{C+1}-s\right)}\right)^e
		\end{align*}
		The series diverges for all $s$ such that $\Re(s)\leq \frac{C - C^2}{C+1} (n-1) + \frac{C-1}{C+1}$. 
\end{proof}

Consider the bound $G(C,e,n)$ used in Lemma \ref{sum_bound_2}. In order to maximize $G(C,e,n)$ as a function in $n$, consider $\max_{0 \leq C \leq 1 } \frac{C - C^2}{C+1} = 3-2\sqrt{2}$. The maximum occurs when $C =\sqrt{2}-1$. Plugging in this value of $C$, we obtain the following corollary.
\begin{cor} \label{sum_bound_3}
Let $n > 1$ and let $G(C,e,n)$ be as in Lemma \ref{sum_bound_2}. Setting $C =  \sqrt{2}-1$, we find that $\sum_{e \ge n-1} p^{G(1-\sqrt{2}, n)}p^{-es}$ diverges for all $s$ such that $\Re(s)\leq (3-2\sqrt{2})(n-1) + 1 - \sqrt{2}$.
\end{cor}

\noindent 
Lemma \ref{sum_bound_1} and Corollary \ref{sum_bound_3} combined give the proof of Theorems \ref{maintheorem_1}. Recall that $$\sum_{e \geq 0}f_n(p^e)p^{-es} \geq \sum_{e \geq n-1} p^{B}p^{-es}$$ for each choice of bound $B$ as above. Therefore the previous lemmas give regions where the local factors of $\zeta_{\Z^n}^R(s)$ diverge. Observe that $f_n(p^e) \geq p^{F(d,e,n)}$ for all integers $d \in [0,n-1]$ and $f_n(p^e) \geq p^{G(C,e,n)}$ for all real numbers $C \in [0,1]$, so we can take a maximum over all $d$ in Lemma \ref{sum_bound_1} or over all $C$ in Lemma \ref{sum_bound_2} to find the largest possible regions of divergence for these geometric series.

It is possible that there are poles further to the right of the ones found above in the given geometric series. Consider the bound in Lemma \ref{sum_bound_1}. When $s > c_7(n)$, $\zeta_{\Z^n}^R(s)$ diverges if 
$$
\sum_p p^{-d(n-1-d)} \cdot \frac{p^{\frac{d(n-1-d)(n-1)}{n-1+d} -(n-1)s}}{ 1-p^{\frac{d(n-1-d)}{n-1+d}-s} }
$$
diverges for all $0 \leq d \leq n-1$. It is a simple computation to show that this series converges on $s  > c_7(n)$. Similar computations show that we cannot find a larger region of divergence for the local factors by using Lemma \ref{sum_bound_2} and Corollary \ref{sum_bound_3} either. Thus Theorem \ref{maintheorem_1} gives the best possible lower bound for the abscissa of convergence of $\zeta_{\Z^n, \, p}^R(s)$ given our lower bounds for $f_n(p^e)$.

\begin{proof}[Proof of Theorem \ref{maintheorem_1}]
Set $$b(n,e) = \bigg\lfloor\frac{e}{n-1+d}\bigg \rfloor d(n-1-d)$$ and $$F(d,e,n) = \frac{ed(n-1-d)}{n-1+d} - d(n-1-d).$$ By Theorem \ref{fn_bound_1}, $f_n(p^e) \ge p^{b(n,e)} \ge p^{F(d,e,n)}$ for each $e \ge n-1$. Therefore 
$$\zeta_{\Z^n, \,p}^R(s) = \sum_{e \ge 0} f_n(p^e)p^{-es} \ge \sum_{e \ge n-1} p^{F(d,e,n)}p^{-es}.$$
 By Lemma \ref{sum_bound_1}, the simpler geometric series diverges for all $s$ such that $\Re(s) \le \frac{d(n-1-d)}{n-1+d}$ and thus $\zeta_{\Z^n, \, p}^R(s)$ diverges on the same region. Taking a maximum over all integers $d \in [0,n-1]$ and applying a Tauberian theorem gives the result.
\end{proof}

The following proposition is strictly worse than Theorem \ref{maintheorem_1} -- it comes from choosing a specific value of $C \in [0,1]$ -- but is easier to use directly.

\begin{prop}
	\label{maintheorem_2} Fix $n > 1$.  Then $\zeta_{\Z^n, \, p}^R(s)$ diverges for all $s$ such that $$ \Re(s)\leq (3-2\sqrt{2})(n-1) + 1 - \sqrt{2}.$$
\end{prop}

\begin{proof}
	The proof is similar to that of Theorem \ref{maintheorem_1}, replacing Lemma \ref{sum_bound_1} with Corollary \ref{sum_bound_3}.
\end{proof}




\section{Orders in a number field} \label{order_sect}
We now study a related zeta function and use results from previous sections to find new lower bounds for the number of orders in a number field. Let $K$ be a number field of degree $n$ with ring of integers $\OO_K$.  Let $$F_K(k) = \#\{\OO \subset \OO_K \, : \, \OO \text{ is a order of } \OO_K \text{ and } |\text{disc}(\OO)| =k\}.$$
Recall that there is a relation between $\text{disc}(\OO)$ and $\text{disc}(\OO_K)$ given by
$$
\text{disc}(\OO) = \text{disc}(\OO_K)[\OO_K:\OO]^2.
$$
 Consider the order zeta function
$$
\eta_K(s) = \sum_{\OO \text{ order of } \OO_K} |\text{disc}(\OO)|^{-s}= \sum_{k=1}^\infty F_K(k)k^{-s}.
$$
This zeta function is closely related to the zeta function
$$
\tilde{\eta}_{K}(s) = \sum_{\OO \text{ order of } \OO_K} [\OO_K :\OO]^{-s} 
$$
by the relation $\eta_K(s) = |\text{disc}(\OO_K)|^{-s} \tilde{\eta}_K(2s).$ 

Notice that $\tilde{\eta}_K(s)$ also has an Euler product $\prod_p \tilde{\eta}_{K, p}(s)$ indexed over the rational primes where
$$
\tilde{\eta}_{K, p}(s) = \sum_{\OO \text{ order of } \OO_K} [\OO_K \otimes_{\Z} \Z_p : \OO]^{-s}.
$$

If $p$ splits completely, then $\OO_K \otimes_{\Z} \Z_p \cong \Z_p^n$. Therefore for all primes $p$ that split completely, $$\tilde{\eta}_{K, p}(s) = {\zeta}_{\Z_p^n}^R(s) =\zeta_{\Z^{n}, \, p}^R(s).$$ Let $N_K(B) = \sum_{X \leq B} F_K(X)$.  The following theorem is due to Kaplan, Marcinek, and Takloo-Bighash \cite{kmt}; see their paper for details on $r_2$, which is a constant that depends on the Galois group of the normal closure of $K/\Q$.
\begin{thm} \cite[Theorem 2]{kmt} \label{kmt_orders}
	\begin{enumerate}
		\item Let $n \leq 5$. Then there exists a constant $C_K  >0$ so that $$N_K(B) \sim C_K B^{\frac{1}{2}}(\log B)^{r_2-1}$$ as $B \rightarrow \infty$.
		\item Let $n > 5$. Then for every $\epsilon > 0$, 
		$$B^{\frac{1}{2}}(\log B)^{r_2-1} \ll N_K(B) \ll_\epsilon B^{\frac{n}{4} - \frac{7}{12} + \epsilon}$$
		as $B \rightarrow \infty$.
	\end{enumerate}
\end{thm}

We can use Theorem \ref{brakenhoff_lb} along with a Tauberian theorem to obtain an improvement on the lower bound in Theorem \ref{kmt_orders}(2).

\begin{thm}
		 Fix $n>1$ and let $$a(n) =  \max_{0 \leq d \leq n-1} \left(\frac{d(n-1-d)}{n-1+d} + \frac{1}{n-1+d}\right).$$ Then
		$
		B^{\frac{1}{2}a(n)} \ll N_K(B)
		$
		as $B \rightarrow \infty$.
	
\end{thm}

As a consequence of Theorem \ref{maintheorem_1} and Proposition \ref{maintheorem_2}, we can also bound the abscissa of convergence of $\tilde{\eta}_{K, \, p}(s)$. 
\begin{thm} Let $n > 1$ be an integer and let $K$ be a degree $n$ number field. Then
	\begin{enumerate}
		\item The zeta function $\tilde{\eta}_{K, \, p}(s)$ diverges for all $s$ such that $\Re(s) \le \frac{c_7(n)}{2}.$
		\item  The zeta function $\tilde{\eta}_{K, \, p}(s)$ diverges for all $s$ such that $$\Re(s) \le \frac{(3-2\sqrt{2})(n-1) +1 -\sqrt{2}}{2} .$$
	\end{enumerate}
\end{thm}
\section{Further questions}\label{further_q}
In this process of bounding $f_n(p^{e})$ by counting irreducible subrings, we made a few assumptions. First, we only considered compositions $\alpha \in C_{n, d, k, \ell}$. Second, we set several entries in the matrix with diagonal $\alpha$ equal to 0 and the rest equal to $p^{\ceil{\frac{k}{2}}} a_{ij}$ for some $a_{ij} \in [0, p^{k - \ceil{\frac{k}{2}}})$. Lastly, we bounded $f_n(p^e)$ by $g_n(p^e)$. These simplifications lead to the following questions.

\begin{question}\label{gn_main_qs}
	Let $d \in [0,n-1]$ be an integer and let $k, \ell$ be positive integers so that $\ell \ge \ceil{\frac{k}{2}}$. Let $\alpha \in C_{n,d,k,\ell}$.
	\begin{enumerate}
		\item  Does the main term of $g_n(p^{e})$ always come from $g_\alpha(p)$ for some $\alpha\in C_{n, d, k, \ell}$?
		\item Is $p^{(k - \ceil{\frac{k}{2}})\text{Area}(P_\alpha)}$ always the main term of $g_\alpha(p)$?
		\item Do $f_n(p^e)$ and $g_n(p^e)$ always have the same main terms?
	\end{enumerate}
\end{question}

The answer to Question \ref{gn_main_qs}(1) is no. For example, Atanasov et al. \cite{akkm} show that $g_5(p^7)$ is a polynomial of degree 4, with the main term coming from the compositions $(3,2,1,1)$ and $(2, 3, 1, 1)$. It is unclear how often pairs $n$ and $e$ are counterexamples to Question \ref{gn_main_qs}(1). It is also unknown how far off the main term of $g_n(p^e)$ can be from the main term of $\max_\alpha g_{\alpha}(p)$ where the maximum is taken over all $\alpha$ of the form above.

It may be the case that a composition of the form $\alpha \in C_{n, d, k, \ell}$ leads to the main term of $g_n(p^e)$, but our lower bound for the number of irreducible subring matrices with diagonal corresponding to $\alpha$ does not give the main term. While the answer to Question \ref{gn_main_qs}(2) is not understood for most pairs $n$ and $e$, we give a partial answer. First, we state some necessary propositions.

 \begin{prop} \cite[Proposition 4.3]{liu} \label{gn_liu} Fix $n > 1$. Then
	\hfill
	\begin{enumerate}
		\item $g_n(p^{n-1}) = 1$
		\item $g_n(p^n) = \frac{p^{n-1}-1}{p-1}.$
	\end{enumerate}
\end{prop}

Corollary 3.7 in \cite{akkm} gives an exact formula for $g_n(p^{n+1})$. To save space, we rewrite their corollary in terms of the degree of $g_n(p^{n+1})$. 

\begin{cor}\cite[Corollary 3.7]{akkm}\label{pn+1}
	Let $n \ge 4$. The function $g_n(p^{n+1})$ is a polynomial in $p$ of degree $2n-6$.
\end{cor}
\begin{example} Let $\alpha\in C_{n, d, 2,1}$. Observe that $e = n-1+d$.
	
	When $d =0$, $e = n-1$. Proposition \ref{gn_liu} shows that $g_n(p^{n-1}) =1$, which matches the bound from Theorem \ref{fn_bound_1}. We see that $g_n(p^{n-1}) = \pbin{n-1}{0}_p$.
	
	When $d=1$, $e=n$. The second part of Proposition \ref{gn_liu} shows that the main term of $g_n(p^{n})$ is $p^{n-2}$. This agrees with Theorem \ref{fn_bound_1}. Further, $g_n(p^n) = \frac{p^{n-1}-1}{p-1}=\pbin{n-1}{1}_p$, so our method counts all possible irreducible subrings of index $p^n$.
	
	 When $d = 2$, $e = n+1$. By Corollary \ref{pn+1}, the main term of $g_n(p^{n+1})$ is $p^{2n-6}$, which matches the main term in Theorem \ref{fn_bound_1}. In this case, our lower bound $g_n(p^{n+1}) \ge \pbin{n-1}{2}_p$ is strictly smaller than the actual formula for $g_n(p^{n+1})$, but the main term is the same. In fact, $$g_n(p^{n+1}) - \pbin{n-1}{2}_p = \binom{n}{2} p^{n-2}.$$
	
\end{example}

Lastly, we provide a partial answer to Part 3 of Question \ref{gn_main_qs}. We do not understand the relationship between the main term of $f_n(p^e)$ and the main term of $g_n(p^e)$ for each fixed $n$ and $e \ge n-1$. In fact, there are classes of examples for which the main term for $f_n(p^e)$ is greater than the main term of $g_n(p^e)$. 

\begin{example} Let $e = n-1$. By Proposition \ref{gn_liu}, $g_n(p^{n-1}) = 1$ for all $n \geq 2$. However, the term $f_1(p^0) g_{n-1}(p^{n-1})$ appears in the recurrence relation stated in Proposition \ref{recurrence}. Proposition \ref{gn_liu} implies that $g_{n-1}(p^{n-1})$ is a polynomial in $p$ with main term $p^{n-3}$. Therefore $f_n(p^{n-1}) \ge p^{n-3}$ whereas $g_n(p^{n-1}) = 1$.
\end{example}


In some cases we are likely not capturing the highest order term of $f_n(p^e)$ by using the bound $f_n(p^e) \ge g_n(p^e)$. However, this is currently the best known approach for counting subrings via subring matrices.

In this paper, we give two new lower bounds for $f_n(p^e)$. Data suggests that the lower bound for $f_n(p^e)$ that comes from counting irreducible subring matrices is slightly worse than the lower bound that comes from counting subgroups. Both of these lower bounds are at most $p^{(3-2\sqrt{2})e(n-1)}$. In order to improve upon the lower bounds given in this paper using these techniques, it seems necessary to answer Questions \ref{gn_main_qs}(2) and \ref{gn_main_qs}(3) or to find other related sets of subgroups that are also subrings. Improvements of the lower bounds for $f_n(p^e)$ would likely lead to better lower bounds for the asymptotic growth of subrings in $\Z^n$ or orders in a fixed number field.


\section{Acknowledgments}
The author thanks Nathan Kaplan for suggesting the problem and for many helpful conversations. The author also thanks the anonymous referee for their helpful comments. This work was supported by the NSF grant DMS 1802281.

	\bibliography{../../../../Bibliography/bib_all}
	\bibliographystyle{abbrv}
\end{document}